\documentclass[12pt, reqno]{amsart}
\usepackage{amsmath,amsopn,amssymb,amsthm}
\usepackage[breaklinks=true,colorlinks=true,linkcolor=blue,citecolor=blue,urlcolor=blue]{hyperref}
\usepackage[cp1251]{inputenc}
\usepackage[english]{babel}

\renewenvironment{proof}[1][Proof]{\textbf{#1.} }
{\ \rule{0.5em}{0.5em}}

\DeclareMathOperator{\ad}{ad}
\DeclareMathOperator{\diag}{diag}
\DeclareMathOperator{\Ric}{Ric}
\DeclareMathOperator{\Lin}{Lin}
\DeclareMathOperator{\End}{End}
\DeclareMathOperator{\Der}{Der}
\DeclareMathOperator{\InnDer}{InnDer}
\DeclareMathOperator{\trace}{trace}

\newtheorem{theorem}{Theorem}
\newtheorem{pred}{Proposition}
\newtheorem{lem}{Lemma}
\newtheorem{cor}{Corollary}

\newtheorem{claim}{Claim}

\newtheorem{remark}{Remark}

\begin{document}

\title[Negative eigenvalues of the Ricci operator \dots ]
{Negative eigenvalues of the Ricci operator \\ of solvable metric Lie algebras}

\author{Yu.~G.~Nikonorov}

\address{Yu.G. Nikonorov \newline
South Mathematical Institute of \newline
Vladikavkaz Scientific Centre of \newline
the Russian Academy of Sciences \newline
Vladikavkaz, Markus st. 22, \newline
362027, RUSSIA}
\email{nikonorov2006@mail.ru}

\begin{abstract}
In this paper we get a necessary and sufficient condition for the Ricci operator of a solvable metric Lie algebra
to have at least two negative eigenvalues.
In particular, this condition implies that the Ricci operator of every non-unimodular
solvable metric Lie algebra or every non-abelian nilpotent
metric Lie algebra has this property.

\vspace{2mm} \noindent Key word and phrases: left-invariant Riemannian metrics, Lie groups, metric Lie algebras,
Ricci operator, eigenvalues of the Ricci operator, Ricci curvature.

\vspace{2mm}

\end{abstract}

\thanks{The project was supported in part by the
State Maintenance Program for the Leading Scientific Schools of
the Russian Federation (grant NSh-921.2012.1) and by Federal Target
Grant ``Scientific and educational personnel of innovative
Russia'' for 2009-2013 (agreement no. 8206, application no. 2012-1.1-12-000-1003-014).}

\maketitle

\section*{Introduction and the main results}

Various restrictions on the curvature of a Riemannian manifold
allow to obtain some interesting information on its geometric and topological structures.
One of the important characteristics of the curvature is the Ricci curvature,
that is confirmed by numerous researches of mathematicians and physicists \cite{Bes}.
On the other hand, it should be noted that
there are many unsolved problems connected with the Ricci curvature,
even in the case of homogeneous Riemannian manifolds (see e.~g. the survey \cite{NikRodSl}
for a more detailed information on this subject).

One of this problem is the following: to classify all possible signatures of the Ricci operators
of invariant Riemannian metrics on a given homogeneous space.
This problem seems to be very hard in general. Now, it is solved only for some very special cases.
There are some important results in this direction \cite{AlKim,Bes,DM}.
For instance, this problem is completely solved
for all homogeneous spaces of dimension $\leq 4$ (see \cite{KrNk1,KrNk2} and references therein).
In particular, J.~Milnor classifies in \cite{Mi} all possible signatures of the Ricci operators
of left-invariant Riemannian metrics on all Lie groups of dimension $\leq 3$,
the same result for Lie groups of dimension $4$ was obtained by
A.G.~Kremlyov and Yu.G.~Nikonorov in \cite{KrNk1,KrNk2}
(some results in this direction are obtained also in the paper \cite{Chen} of D.~Chen).

For other dimensions we have only partial results. It is necessary to mention
the paper \cite{DM} of I.~Dotti-Miatello, where Ricci signatures of left-invariant
Riemannian metrics on two-step solvable unimodular Lie groups are determined,
and the paper \cite{Krem} of A.G.~Kremlyov, where the same problem is solved for
nilpotent five-dimensional Lie groups.

In this paper we restrict our attention to solvable Lie groups with left-invariant Riemannian metrics.
It is shown in the papers \cite{Jensen,Mi} that the scalar curvature of
every non-flat left-invariant Riemannian metric on a given solvable
Lie group is negative, therefore, the Ricci operator of this metric has at least
one negative eigenvalue.
The main problem of this paper is to determine, {\it whether the Ricci operator
of a given left-invariant metric has at least two negative eigenvalues.}

It is convenient to study
left-invariant Riemannian metrics on Lie groups in terms of metric Lie algebras
(i.~e. Lie algebras supplied with inner products)
\cite{Al5, Bes, NikRodSl}.
Indeed, let $G$ be a Lie group with the Lie algebra $\mathfrak{g}$.
Then every inner product $(\cdot ,\cdot)$ on $\mathfrak{g}$ uniquely
determines a left-invariant Riemannian metric $\rho$ on $G$, and
vice versa (see e.~g. 7.24 in \cite{Bes}). As usual, we denote by $[\mathfrak{s},\mathfrak{s}]$
the derived algebra of a Lie algebra~$\mathfrak{s}$. For every solvable Lie algebra $\mathfrak{s}$,
$[\mathfrak{s},\mathfrak{s}]$ is a {\it nilpotent ideal} of~$\mathfrak{s}$ and
$[\mathfrak{s},\mathfrak{s}]\neq \mathfrak{s}$.

\smallskip

Recall, that an operator $A$ (acting on a given Euclidean space) is called {\it normal}, if it commutes
with its adjoint $A'$.
The main result of this paper is the following

\begin{theorem}\label{VRazAlg}
Let $(\mathfrak{s},Q)$ be a solvable metric Lie algebra, $\mathfrak{n}=[\mathfrak{s},\mathfrak{s}]$, $\mathfrak{a}$ be
a $Q$-orthogonal complement to $\mathfrak{n}$ in $\mathfrak{s}$.
Then one of the following mutually exclusive assertions holds:

\begin{enumerate}

\item The ideal $\mathfrak{n}$ is commutative, $\mathfrak{a}$ is  a commutative subalgebra of $\mathfrak{s}$,
and for every $X\in \mathfrak{a}$
the operator $\ad(X)|_{\mathfrak{n}}$ is skew-symmetric with respect to $Q$
(in this case the Ricci operator of $(\mathfrak{s},Q)$ is zero);
\item The ideal $\mathfrak{n}$ is commutative, $\mathfrak{a}$ is  a commutative subalgebra of $\mathfrak{s}$,
and for every $X\in \mathfrak{a}$
the operator $\ad(X)|_{\mathfrak{n}}$ is trace-free and normal with respect to $Q$, but the subspace
$$
\mathfrak{b}=\{X\in \mathfrak{a}\,|\, \ad(X)|_{\mathfrak{n}}\,\, \mbox{is skew-symmetric with respect to}\,\, Q\}
$$
has codimension $1$ in $\mathfrak{a}$
(in this case the Ricci operator of $(\mathfrak{s},Q)$
has only one negative eigenvalue, while all other eigenvalues are zero);
\item The Ricci operator of the metric Lie algebra $(\mathfrak{s},Q)$ has at least two negative eigenvalues.
\end{enumerate}
\end{theorem}

\begin{remark}\label{calcul}
The structure of the Ricci operator in the items (1) and (2) of the above theorem easily follows from
the formula (\ref{nnteor}).  Metric Lie algebras with zero Ricci curvature are flat
(i.~e. have zero sectional curvature) by the well known result
of D.V.~Alekseevski\v\i~and B.N.~Kimmel'fel'd~\cite{AlKim}. The case of flat metric Lie algebras
has been studied in Theorem 1.5 of the paper~\cite{Mi} by J.~Milnor.
\end{remark}

The authors of the paper \cite{KrNk2} proved (in particular) that the Ricci operator
of every non-unimodular solvable metric Lie algebra of dimension $\leq 4$ has at least two negative
eigenvalues. Moreover, it has been conjectured in \cite{KrNk2}, that the Ricci operator of
an arbitrary non-unimodular solvable metric Lie algebra
have the same property. This conjecture was confirmed for
all non-unimodular solvable metric Lie algebras of dimension $\leq 6$ in \cite{Cheb1},
for all completely solvable Lie algebras in~\cite{NikCheb}, and
for all Lie algebras with six-dimensional two-step nilpotent derived algebras in~\cite{Abiev}.
\smallskip

Obviously, the cases (1) and (2) of Theorem \ref{VRazAlg}
are impossible for non-unimodular Lie algebras. Hence, Theorem \ref {VRazAlg} implies
immediately a confirmation  of the above-mentioned conjecture:

\begin{theorem}
\label{conject}
Let $\mathfrak{s}$ be a non-unimodular solvable Lie algebra.
Then for every inner product $Q$ on $\mathfrak{s}$, the Ricci operator of the metric
Lie algebra $(\mathfrak{s}, Q)$ has at least two negative eigenvalues.
\end{theorem}

Using Theorem \ref{VRazAlg}, we can get the following result (see Remark \ref{last} in the last section).

\begin{theorem}[\cite{NikCheb}]\label{NilpAlg}
Let $\mathfrak{s}$ be a non-commutative nilpotent Lie algebra.
Then for every inner product $Q$ on $\mathfrak{s}$ the Ricci operator of the metric
Lie algebra $(\mathfrak{s}, Q)$ has at least two negative eigenvalues.
\end{theorem}

Note that some partial cases of this theorem were obtained earlier in the paper~\cite{Krem}.

From Theorem \ref{VRazAlg} we easily get also the following two corollaries.

\begin{cor}\label{sled1}
If the Ricci operator of a solvable metric Lie algebra $(\mathfrak{s}, Q)$ has at least one positive eigenvalue, then
it has at least two negative eigenvalues.
\end{cor}

\begin{cor}\label{sled2}
Let $(\mathfrak{s}, Q)$ be a solvable metric Lie algebra such that
a $Q$-orthogonal complement $\mathfrak{a}$ to
$\mathfrak{n}=[\mathfrak{s},\mathfrak{s}]$ in $\mathfrak{s}$
is not a commutative subalgebra of $\mathfrak{s}$, then
the Ricci operator of $(\mathfrak{s}, Q)$ has at least two negative eigenvalues.
\end{cor}

We hope that results of this paper will be useful for future research on solvable metric Lie algebras
(therefore, on solvable Lie groups with left-invariant Riemannian metrics), in particular,
for the study of the Ricci flow
on solvable Lie groups (see \cite{GliPay, La2011, Pay2010}).

\medskip

The structure of this paper is the following.
In the first section we recall some notations and useful facts and prove also some auxiliary results.
The second section is devoted to some convenient formulas for the Ricci operator of solvable metric Lie algebras.
In this section, we recall also some important results related to the Ricci curvature.
The third section of the paper is devoted to the proof of Theorem
\ref{VRazAlg} for one special (but very involved and important) case.
Finally, in the last section we prove Theorem~\ref{VRazAlg} in full generality.
It should be noted, that we have used Theorem~\ref{NilpAlg} for this goal,
but (for a formal reason) we have proved completely this theorem  before its using.
Hence, our presentation does not depend on the paper~\cite{NikCheb}.
\medskip

The author is indebted to Prof. V.N.~Berestovskii and Prof. Yu.A.~Nikolayevsky
for helpful discussions concerning this paper.

\section{Notations and auxiliary results}

Standard notations and classical results on Lie algebras could be find in \cite{Bour, Ja, Vinb}.

Let $\mathfrak{n}$ be a nilpotent Lie algebra of degree $p$. Consider its lower central series $\{\mathfrak{n}^k \}$,
where
$$
\mathfrak{n}^0=\mathfrak{n}, \quad \mathfrak{n}^1=[\mathfrak{n},\mathfrak{n}], \quad \dots ,\quad
\mathfrak{n}^k =[\mathfrak{n}^{k-1},\mathfrak{n}]\quad (k \geq 1),\quad \dots
$$
Then $\mathfrak{n}^p=0$ and  $\mathfrak{n}^{p-1}\neq 0$.

Let $\Der(\mathfrak{n})$ and $\InnDer(\mathfrak{n})$ be a space of derivations and a space of inner derivations
of the Lie algebra $\mathfrak{n}$ respectively. It is clear that $\InnDer(\mathfrak{n})\subset \Der(\mathfrak{n})$ and
$\InnDer(\mathfrak{n})\neq \Der(\mathfrak{n})$ since $\mathfrak{n}$ is nilpotent \cite{Ja}.

\begin{lem}\label{derstruc}
For any $A \in \Der(\mathfrak{n})$ we have $A(\mathfrak{n}^k) \subset \mathfrak{n}^k$ for every $k \geq 0$.
\end{lem}

\begin{proof}
We prove the lemma by induction. For $k=0$ we get $A(\mathfrak{n}^0)=A(\mathfrak{n})=\mathfrak{n}=\mathfrak{n}^0$.
If the lemma is true for all values of $k\leq l$ then (by properties of derivations) we have
$A(\mathfrak{n}^{l+1})=A([\mathfrak{n}^{l},\mathfrak{n}])\subset
[A(\mathfrak{n}^{l}),\mathfrak{n}]+[\mathfrak{n}^{l}, A(\mathfrak{n})]\subset
[\mathfrak{n}^{l},\mathfrak{n}]+[\mathfrak{n}^{l},\mathfrak{n}]\subset \mathfrak{n}^{l+1}$.
This proves the lemma.
\end{proof}

\begin{lem}\label{innderstruc}
For any $A \in \InnDer(\mathfrak{n})$ we have $A(\mathfrak{n}^k) \subset \mathfrak{n}^{k+1}$ for every $k \geq 0$.
\end{lem}

\begin{proof}
Since $A \in \InnDer(\mathfrak{n})$, then there is $X \in \mathfrak{n}$ such that $A=\ad (X)$.
For every $k$ we get $A(\mathfrak{n}^k)=[X, \mathfrak{n}^k] \subset [\mathfrak{n}^k, \mathfrak{n}]=\mathfrak{n}^{k+1}$.
This proves the lemma.
\end{proof}

\begin{lem}\label{vspom0}
Let  $\mathfrak{n}$ be a nilpotent Lie algebra, $\mathfrak{b}\subset \mathfrak{n}$ is an abelian subalgebra
of codimension $1$. Then $\mathfrak{b}$ is an abelian ideal in $\mathfrak{n}$.
\end{lem}

\begin{proof} One can find the proof of this lemma e.~g. in \cite{Ten}, but we give it here
for the convenience of the reader.
Fix some $Y \in \mathfrak{b}$. We should prove  that $[Y,X]\in \mathfrak{b}$ for every $X\in \mathfrak{n}$.
If $X\in  \mathfrak{b}$, then $[Y,X]=0$. Now we suppose that $X \not\in  \mathfrak{b}$.
Obviously, there are $\alpha \in \mathbb{R}$ and $Z\in \mathfrak{b}$ such that
$\ad(Y)(X)=[Y,X]=\alpha X+Z$. Then we have $\ad^2(Y)(X)=[Y, \alpha X+Z]=\alpha^2 X+ \alpha Z$, $\dots$,
$\ad^k (Y)(X)=\alpha^k X+ \alpha^{k-1} Z$, $k \geq 1$. Since $\mathfrak{n}$ is nilpotent, then
$\ad^k (Y)(X)=0$ for sufficiently large $k$. Therefore, $\alpha=0$ and
$[Y,X]=Z \in \mathfrak{b}$.
\end{proof}

\medskip

Now, we fix {\it an inner product $(\cdot,\cdot)$ on the Lie algebra $\mathfrak{n}$} and consider some
important properties of
{\it the metric Lie algebra $(\mathfrak{n},(\cdot,\cdot))$}.
Let $V_k$ be a $(\cdot,\cdot)$-orthogonal complement to $\mathfrak{n}^k$ in $\mathfrak{n}^{k-1}$, $k=1,2,\dots,p$.
It is clear that $V_p=\mathfrak{n}^{p-1}$, since $\mathfrak{n}^{p}=0$.

\medskip

We consider the space $\End(\mathfrak{n})$ of linear endomorphisms of the Lie algebra $\mathfrak{n}$ and
its three subspaces $L_1$, $L_2$ and $L_3$, defined as follows:
\begin{eqnarray*}
L_1=\{A \in \End(\mathfrak{n})\,|\,A(\mathfrak{n}^k)\subset \mathfrak{n}^k \,\, \mbox{for} \,\,k=0,1,\dots, p-1\},\hspace{3mm}
\\
L_2=\{A \in \End(\mathfrak{n})\,|\,A(V_k)\subset V_k \,\, \mbox{for} \,\,k=1,\dots, p\},\hspace{14mm}
\\
L_3=\{A \in \End(\mathfrak{n})\,|\,A(\mathfrak{n}^k)\subset \mathfrak{n}^{k+1} \,\, \mbox{for} \,\,k=0,1,\dots, p-1\},
\end{eqnarray*}

Obviously, Lemmas \ref{derstruc} and
\ref{innderstruc} imply the following

\begin{cor}\label{subder}
For every nilpotent Lie algebra $\mathfrak{n}$ the inclusions $\Der(\mathfrak{n})\subset L_1$ and
$\InnDer(\mathfrak{n})\subset L_3$ hold.
\end{cor}

\smallskip

In what follows, we denote by {\it $C'$ the adjoint operator to the operator
$C\in\End(\mathfrak{n})$ with respect to the inner product $(\cdot, \cdot)$ on $\mathfrak{n}$},
i.~e. $(C(X),Y)=(X,C'(Y))$ for every $X,Y \in \mathfrak{n}$. If we represent operators from $\End(\mathfrak{n})$
by matrices in some $(\cdot, \cdot)$-orthonormal basis for $\mathfrak{n}$,
then {\it $C'$ is the transpose of the matrix $C$}. We will use also the notation $C^s$ for a
symmetric part of the operator $C$ with respect to $(\cdot, \cdot)$, i.~e. $C^s=\frac{1}{2}(C+C')$.
\smallskip

\begin{lem}\label{vspom1n}
If $A \in \Der(\mathfrak{n})$ is such that $A' \in \Der(\mathfrak{n})$,
then $A,A' \in L_2$.
\end{lem}

\begin{proof}
Consider any $X\in \mathfrak{n}^k$ and any $Y \in V_k$, $k=1,\dots, p$.
By Corollary~\ref{subder} we get $A, A' \in L_1$. Therefore, $A(\mathfrak{n}^k)\subset \mathfrak{n}^k$,
$A'(\mathfrak{n}^k)\subset \mathfrak{n}^k$,
$A(\mathfrak{n}^{k-1})\subset \mathfrak{n}^{k-1}$ and $A'(\mathfrak{n}^{k-1})\subset \mathfrak{n}^{k-1}$.
Since $V_k$ is the orthogonal
complement to $\mathfrak{n}^k$ in $\mathfrak{n}^{k-1}$, then we have
$$
(X,A(Y))=(A'(X),Y)=0, \quad (X,A'(Y))=(A(X),Y)=0.
$$
This implies $A(V_k) \subset V_k$ and $A'(V_k) \subset V_k$ for all $k=1,\dots, p$, i.~e. $A,A' \in L_2$.
\end{proof}

\medskip

{\it We supply the linear space $\End(\mathfrak{n})$ with an inner product $\langle \cdot, \cdot \rangle$ as follows}:
$\langle A, B \rangle =\trace (AB')$, where  $B'$ the adjoint operator to the operator
$B\in\End(\mathfrak{n})$ with respect to $(\cdot, \cdot)$.
Using a matrix representation in any $(\cdot, \cdot)$-ortho\-normal basis of $\mathfrak{n}$,
we easily get
$\langle A, B \rangle =\trace (AB')=\trace (B'A)=\trace(A'B)=\trace (BA')$.
Note also that
$\langle A', B' \rangle=\trace(A'B)=\trace(BA')=\langle A, B \rangle$ for every $A,B \in \End(\mathfrak{n})$.

\smallskip

\begin{lem}\label{vspom2}
If $A \in L_2$ and $B \in L_3$, then $\langle A, B \rangle=0$.
\end{lem}

\begin{proof}
Since $(A(V_k),V_l)=0$ is equivalent to $(V_k, A'(V_l))=0$, then $A' \in L_2$. Further, since
$BA'(V_k)\subset B(V_k)\subset B(\mathfrak{n}^{k-1})\subset \mathfrak{n}^{k}$
and $BA'(\mathfrak{n}^{k-1})=BA'\bigl(\bigoplus_{i\geq k}{V_i}\bigr) \subset \mathfrak{n}^{k}$, then we get
$BA'\in L_3$.
But every operator
$C\in L_3$ is trace-free (it is easy to check by using a basis $\{e_j\}$ in $\mathfrak{n}$, such that
every $e_j$ lies in some $V_l$).
Therefore, $\langle A, B \rangle=\trace(BA')=0$.
\end{proof}

\begin{lem}\label{vnilp}
For any nilpotent matrix $L=(l_{ij})$
with real entries the equality
$$
2\trace(L^s \cdot L^s) = \trace(L \cdot L')
$$
holds, where
$L'$ is the transpose of the  matrix $L$, $L^s=\frac{1}{2}(L+L')$ is a symmetric part of~$L$.
\end{lem}

\begin{proof} For any orthogonal matrix $Q$,
the above equality does not change when we replace $L$ with $QLQ^{-1}$.
Hence, it suffices to consider the case when $L$ is upper triangular with zeros on the main diagonal.
Then we get
$$
4\trace(L^s \cdot L^s)=\sum\limits_{i,j} (l_{ij}+l_{ji})^2=
2\sum\limits_{i,j} l_{ij}^2+2\sum\limits_{i,j} l_{ij}l_{ji}=2\sum\limits_{i,j} l_{ij}^2=2\trace(L \cdot L').
$$
This proves the lemma.
\end{proof}

\medskip

Obviously,
$\InnDer(\mathfrak{n})\subset \Der(\mathfrak{n})\subset \End(\mathfrak{n})$. {\it We need the
$\langle \cdot, \cdot \rangle$-orthogonal projection}
\begin{equation}\label{eqproj}
P_{inner}: \Der(\mathfrak{n}) \rightarrow \InnDer(\mathfrak{n}),
\end{equation}
i.~e. $P_{inner}(A) \in \InnDer(\mathfrak{n})$ and
$\langle A-P_{inner}(A), B \rangle=0$ for every $A \in \Der(\mathfrak{n})$ and $B\in \InnDer(\mathfrak{n})$.

\smallskip

\begin{lem}\label{simple}
Consider any $A\in \Der(\mathfrak{n})$ and put $\widetilde{A}:=P_{inner}(A)$.
Then the equalities
$\trace\left(\widetilde{A}\widetilde{A}\right)=0$ and
$\trace\left(\widetilde{A}\widehat{A}\right)=0$ hold, where $\widehat{A}=A-\widetilde{A}$.
\end{lem}

\begin{proof} Let us consider any $A\in \Der(\mathfrak{n})$.
We define a Lie multiplication $[\cdot, \cdot]_1$ on the direct sum of linear spaces
$\mathfrak{s}:=\mathfrak{n}\oplus \mathbb{R}$ as follows:
$$
[(a_1, b_1), (a_2,b_2)]_1=([a_1,a_2]+b_1A(a_2)-b_2A(a_1),0).
$$
Obviously, $(\mathfrak{s}, [\cdot, \cdot]_1)$ is a solvable Lie algebra, and
$\mathfrak{n}$ (we identify each element $X\in \mathfrak{n}$ with $(X,0)\in \mathfrak{s}$)
lies in the nilradical $\mathcal{N}(\mathfrak{s})$ of $\mathfrak{s}$.
Consider $Y\in \mathfrak{n}$ such that $\ad (Y)=\widetilde{A}$.
Then $\ad \bigl((Y,0)\bigr)|_{\mathfrak{n}}=\widetilde{A}$ and $\ad \bigl((0,1)\bigr)|_{\mathfrak{n}}=A$.
We know that the Killing form $B_{\mathfrak{s}}$ of any solvable Lie algebra $\mathfrak{s}$ satisfies the equation
$B_{\mathfrak{s}}\bigl(\mathfrak{s}, \mathcal{N}(\mathfrak{s})\bigr)=0$
(see e.~g. Remark after Proposition 6 of I.5.5 in \cite{Bour}).
Since $(Y,0)\in \mathcal{N}(\mathfrak{s})$, we get
$$
0=B_{\mathfrak{s}}\bigl((Y,0),(Y,0)\bigr)=\trace\Bigl(\ad \bigl((Y,0)\bigr) \cdot \ad \bigl((Y,0)\bigr)\Bigr)=
$$
$$
\trace \Bigl(\ad \bigl((Y,0)\bigr)|_{\mathfrak{n}}\cdot \ad \bigl((Y,0)\bigr)|_{\mathfrak{n}}\Bigr)=
\trace(\widetilde{A}\widetilde{A})
$$
and
$$
0=B_{\mathfrak{s}}\bigl((0,1),(Y,0)\bigr)=\trace\Bigl(\ad \bigl((0,1)\bigr) \cdot \ad \bigl((Y,0)\bigr)\Bigr)=
$$
$$
\trace\Bigl(\ad \bigl((0,1)\bigr)|_{\mathfrak{n}}\cdot \ad \bigl((Y,0)\bigr)|_{\mathfrak{n}}\Bigr)=\trace(A\widetilde{A}).
$$
As a simple corollary, we get also
$\trace \left(\widetilde{A}\widehat{A}\right)=\trace \left(\widetilde{A}A\right)-
\trace \left(\widetilde{A}\widetilde{A}\right)=0$.
\end{proof}

\begin{pred}\label{vspom1}
For every $A \in \Der(\mathfrak{n})$ the inequality
$$
\langle A^s, A^s \rangle \geq \frac{1}{2} \langle \widetilde{A}, \widetilde{A} \rangle
$$
holds, where $\widetilde{A}=P_{inner}(A)$, $A^s=\frac{1}{2}(A+A')$.
Moreover, $\langle A^s, A^s \rangle = \frac{1}{2} \langle \widetilde{A}, \widetilde{A} \rangle$ if and only if
$A-\widetilde{A}$ is a skew-symmetric derivation of $\mathfrak{n}$.
\end{pred}

\begin{proof}
Put $\widehat{A}:=A-\widetilde{A}$. Then $A=\widetilde{A}+\widehat{A}$,
$2A^s=\widetilde{A}+\widetilde{A}'+2\widehat{A}^s$ and
$$4\langle A^s,A^s \rangle=2\langle \widetilde{A}, \widetilde{A}\rangle+
4\langle \widetilde{A}, \widehat{A}^s\rangle+
4\langle \widetilde{A}', \widehat{A}^s\rangle+4\langle \widehat{A}^s,\widehat{A}^s \rangle,
$$
since
$\langle \widetilde{A}', \widetilde{A}'\rangle=\trace(\widetilde{A}'\widetilde{A})=\trace(\widetilde{A}\widetilde{A}')=
\langle \widetilde{A}, \widetilde{A}\rangle$ and
$\langle \widetilde{A}, \widetilde{A}'\rangle=\trace(\widetilde{A}\widetilde{A})=0$ by Lemma \ref{simple}.
Further, since $\langle \widetilde{A}', \widehat{A}\rangle=\langle \widetilde{A}, \widehat{A}'\rangle=
\trace(\widetilde{A}\widehat{A})=0$ (by Lemma \ref{simple}) and
$\langle \widetilde{A}', \widehat{A}'\rangle=\langle \widetilde{A}, \widehat{A}\rangle=0$
(by definitions of $\widetilde{A}$ and $\widehat{A}$), then $\langle \widetilde{A}, \widehat{A}^s\rangle=
\langle \widetilde{A}', \widehat{A}^s\rangle=0$. Therefore,
$$\langle A^s,A^s \rangle=\frac{1}{2}\langle \widetilde{A}, \widetilde{A}\rangle+
\langle \widehat{A}^s,\widehat{A}^s \rangle \geq \frac{1}{2}\langle \widetilde{A}, \widetilde{A}\rangle.
$$
Clear, that $\langle \widehat{A}^s,\widehat{A}^s \rangle=0$ if and only if the derivation $\widehat{A}$
is skew-symmetric.
\end{proof}

\smallskip

We will need also one well known result on localization of the eigenvalues of a symmetric matrix
and one its obvious corollary.

\begin{pred}[cf. Theorem 4.3.8 in \cite{Horn}]
Let $\widetilde{A}$ be a symmetric
$(n \times n)$-matrix with real entries, $A$ be a matrix obtained from $\widetilde{A}$ by deleting
the last row and the last column. Assume that the eigenvalues
$\lambda_i$ of $A$ and the eigenvalues $\widetilde{\lambda}_i$ of $\widetilde{A}$
have been arranged in increasing order
$\lambda_1 \leq  \lambda_2 \leq \cdots \leq \lambda_{n-2}\leq \lambda_{n-1}$ and
$\widetilde{\lambda}_1 \leq \widetilde{\lambda}_2 \leq \cdots
\leq \widetilde{\lambda}_{n-1} \leq \widetilde{\lambda}_{n}$.
Then the inequality
$$
\widetilde{\lambda}_1 \leq \lambda_1 \leq \widetilde{\lambda}_2 \leq \cdots \leq \widetilde{\lambda}_{n-1}
\leq \lambda_{n-1} \leq \widetilde{\lambda}_{n}
$$
holds.
\end{pred}

\begin{cor}
\label{dopmat}
Let $A$ be a symmetric $(n \times n)$-matrix with real entries, $B$ be a $(m \times m)$-matrix obtained
from $A$ by deleting of $(n-m)$ rows and $(n-m)$ columns with coincided sets of indexes.
If the matrix $B$ is positive (negative) definite, then the matrix $A$ has at least $m$
positive (respectively, negative) eigenvalues.
\end{cor}

\section{The Ricci operator}

Consider a Lie algebra $\mathfrak{g}$ supplied with an inner product
$(\cdot, \cdot)$. We choose some $(\cdot,
\cdot)$-orthonormal basis $\{X_i\}$, $1 \leq i \leq \dim(\mathfrak{g})$, in $\mathfrak{g}$.
Define a vector
$H \in \mathfrak{g}$ by the equality  $(H,X)=\trace (\ad (X))$, where
$\ad (X) (Y)=[X,Y]$, $X, Y\in \mathfrak{g}$.
Note that
$H=0$ if and only if the Lie algebra $\mathfrak{g}$ is unimodular.
For the Ricci operator of the metric Lie algebra
$(\mathfrak{g},(\cdot, \cdot))$ we have the following formula:
\begin{equation}
\label{riccAl}
\operatorname{Ric} = -\frac{1}{2}
\sum\limits_i
{\ad}^{\prime}(X_i)
{\ad} (X_i) +
\frac{1}{4}
\sum\limits_i
{\ad}(X_i)
{\ad}^{\prime}(X_i)
-\frac{1}{2}B - {\ad}^{s}(H),
\end{equation}
where $B$ is the Killing operator, ${\ad}^{\prime}(X_i)$
is the adjoint operator for~${\ad} ({X_i})$ with respect to $(\cdot , \cdot)$, and
${\ad}^{s}(H) = \frac{1}{2}({\ad}(H) +{\ad}^{\prime}(H))$~is a symmetric part
of the operator ${\ad} (H)$~\cite{Al5}.

By $\Ric(X,Y)$ we denote $(\Ric X,Y)=(X, \Ric Y)$, i.~e. the value of the Ricci form on the vectors
$X,Y$ \cite{Al5, Bes}.

\smallskip

Now, we (using some ideas from the paper \cite {NikNik}) get some refinement of the formula~(\ref{riccAl})
for solvable metric Lie algebras.
We will use a notation $M'$ for the transpose of a matrix $M$.

Suppose that a solvable Lie algebra $\mathfrak{s}$ is supplied with an inner product~$Q$.
We are interested in the structures of metric Lie algebras
$(\mathfrak{s},Q)$ and $(\mathfrak{n}, Q|_\mathfrak{n})$,
where $\mathfrak{n}:=[\mathfrak{s},\mathfrak{s}]$~is a derived algebra of the Lie algebra $\mathfrak{s}$.
Let $\mathfrak{a}$ be the orthogonal complement to $\mathfrak{n}$ in
$\mathfrak{s}$ with respect to $Q$.
Put $l =\dim(\mathfrak{n})$ and $m =\dim(\mathfrak{a})$.

Let us choose vectors $\{e_i\}$,  $1 \leq i \leq l$, that form a $Q$-orthonormal basis in $\mathfrak{n}$.
This basis could be completed with a $Q$-orthonormal basis $\{f_1,f_2,...,f_m\}$  in
$\mathfrak{a}$ such that
$$
t:=\trace(\ad(f_1))=t\geq 0,  \quad  \trace(\ad(f_j))=0,  \quad  2 \leq j \leq m.
$$
It is easy to see that for a non-unimodular Lie algebra $\mathfrak{s}$ we have
$f_1=\frac{H}{\|H\|}$, where the vector  $H\in \mathfrak{s}$ is defined by the equation
$Q(H,X) =\trace(\ad(X))$
for all $X \in \mathfrak{s}$. In this case $t=\trace(\ad(f_1))=\|H\|>0$.
If $\mathfrak{s}$ is unimodular, then we can choose any unit vector from $\mathfrak{a}$ as $f_1$.
In this case we get $t=\trace(\ad(f_1))=0$.

It is clear that $\ad(f_j)|_\mathfrak{n} \in \Der(\mathfrak{n})$, $1 \leq j \leq m$, where
$\Der(\mathfrak{n})$ is the Lie algebra of all derivations of $\mathfrak{n}$.
We use the basis $\{e_1,...,e_l, f_1,...,f_m\}$ in order to represent all operators
$\ad(f_j)$ and $\ad(e_i)$ in the matrix form:
\begin{equation}\label{adad}
\ad(f_j) = \left(
\begin{array}{cc}
 A_j  & B_j \\
 0  & 0 \\
 \end{array}
 \right), \quad
\ad(e_i) = \left(
\begin{array}{cc}
 D_i  & C_i\\
 0 & 0 \\
 \end{array}
 \right),
\end{equation}
for some $(l \times l)$-matrices $A_j$, $D_i$  and some $(l \times m)$-matrices $B_j$, $C_i$.

In the basis $\{e_1,...,e_l,f_1,...,f_m\}$, the matrix of the Ricci operator of the solvable metric Lie algebra
$(\mathfrak{s}, Q)$ has the following form (see the formula (\ref{riccAl}) or the proof of Theorem 3 in \cite{NikNik}):
\begin{equation}
\label{nnteor}
\Ric = \left( {{\begin{array}{*{20}c}
 R_1 \hfill & R_2 \hfill\\
 R'_2 \hfill & R_3 \hfill\\
 \end{array} }} \right),
\end{equation}
where
$$
R_1 = \Ric^\mathfrak{n} + \frac{1}{2} \sum\limits_{j=1}^{m}[A_j,A_j'] +
\frac{1}{4} \sum\limits_{j=1}^{m}B_jB_j' - t A_1^s, $$
$$
R_2 = -\frac{1}{2} \left( \sum\limits_{i=1}^{l}D'_iC_i +\sum\limits_{j=1}^{m} A_j'B_j +tB_1 \right),
$$
$$
R_3 = -\frac{1}{2}\sum\limits_{j=1}^{m}B_j'B_j-L,
$$
where $\Ric^\mathfrak{n}$ is the matrix of the Ricci operator of the metric Lie algebra
$(\mathfrak{n}, Q|_{\mathfrak{n}})$
in the basis $\{e_1,....,e_l\}$,
$L$ is a $(m \times m)$-matrix with elements $l_{pq}= \trace(A_p^sA_q^s)$,
$[A_j,A_j'] = A_j A_j' - A_j' A_j$,
$A_j^s = \frac{1}{2} (A_j'+A_j)$, $t=\trace(A_1)=\trace(A_1^s) \geq 0$.

Note also that the formula (\ref{riccAl}) could be simplified for the metric Lie algebra
$(\mathfrak{n}, Q|_{\mathfrak{n}})$
(a nilpotent Lie algebra $\mathfrak{n}$ is unimodular and has trivial Killing form). Namely, we get
the following formula for the matrix of its Ricci operator in the basis $\{e_1,...,e_l\}$:
\begin{equation}
\label{riccAln}
\Ric^\mathfrak{n} = -\frac{1}{2}
\sum\limits_{i=1}^{l} D_i^{\prime}D_i +\frac{1}{4} \sum\limits_{i=1}^{l}D_iD_i^{\prime}.
\end{equation}

We will need the following result.

\begin{pred}[\cite{Mi}]
\label{kossim}
Let $(\mathfrak{g}, (\cdot , \cdot))$ be a metric Lie algebra, $X \in \mathfrak{g}$
is orthogonal to the ideal $[\mathfrak{g}, \mathfrak{g}]$. Then the inequality $\Ric(X,X)\leq 0$ holds.
Moreover, this inequality becomes an equality if and only if the operator $\ad(X)$
is skew-symmet\-ric with respect to $(\cdot , \cdot)$.
\end{pred}

\begin{remark}
Note that this proposition could be easy derived from the formula (\ref{nnteor}). Indeed,
the matrix $R_3 = -\frac{1}{2}\sum\limits_{j=1}^{m}B_j'B_j-L$ is negative semi-definite.
If $X=\sum_{j=1}^m {\alpha}_j f_j$, then $\Ric(X,X)=\Ric(\sum_{j=1}^m {\alpha}_j f_j,\sum_{j=1}^m {\alpha}_j f_j)=
\sum_{i,j} {\alpha}_i {\alpha}_j \Ric (f_i,f_j) \leq 0$. The case $\Ric(X,X)=0$ could be easily studied.
\end{remark}

We will need also the following remarkable property of nilpotent metric Lie algebras.

\begin{pred}[Corollary 5 in \cite{Nik2007}]\label{riccider}
Let $(\mathfrak{n}, (\cdot, \cdot))$ be a nilpotent metric Lie algebra.
Then for every derivation $A\in \Der (\mathfrak{n})$ the inequality
$$
\trace(\Ric^{\mathfrak{n}}\cdot [A,A'])=\langle \Ric^{\mathfrak{n}}, [A,A']\rangle \geq 0
$$
holds, where $\Ric^{\mathfrak{n}}$ is the Ricci operator of  $(\mathfrak{n}, (\cdot, \cdot))$.
Moreover, this inequality becomes an equality if and only if $A'\in \Der (\mathfrak{n})$.
\end{pred}

\section{One important partial case}

The most difficult case in Theorem \ref{VRazAlg} is the case, when the derived algebra
$\mathfrak{n}$ has codimension $1$ in the Lie algebra $\mathfrak{s}$.
We supply the space $\End(\mathfrak{n})$
with the inner product $\langle A,B \rangle =\trace (AB')$, where $B'$ means the adjoint of the operator $A$
with respect to~$Q|_{\mathfrak{n}}$.

At first, we refine the formula (\ref{nnteor}) to this special case.
We have $m=1$ and we will use notations $f$ and $A$ instead of $f_1$ and $A_1$ respectively (see (\ref{adad})).
Obviously, $B_1$ is trivial because of $m=1$.
In the formula (\ref{nnteor}) for the matrix of the Ricci operator, we get
$$
R_1 = \Ric^\mathfrak{n} + \frac{1}{2}[A,A'] - t A^s, \quad
R_2 = -\frac{1}{2}( \sum\limits_{i=1}^{l}D'_iC_i),
$$
and the matrix $R_3$ consists of a unique element  $-r$, where
\begin{equation}\label{impr}
r = \trace(A^s A^s)=\langle A^s, A^s \rangle.
\end{equation}

\begin{remark}\label{eqeq}
It is easy to check that the $i$-th entry
of the column matrix $R_2$ is equal to
$$
-\frac{1}{2} \trace(D'_i \cdot A)=-\frac{1}{2} \trace(D_i \cdot A')=-\frac{1}{2} \langle D_i, A \rangle
$$
(see \cite{Cheb1} for details).
\end{remark}

Now, we consider the structural constants $C_{ij}^k$ of the Lie algebra $\mathfrak{n}$
with respect to the basis $\{e_1, \dots, e_l\}$, i.~e. $[e_i,e_j]=\sum_{k=1}^l C_{ij}^k e_k$ for all $i,j,k$.
It is clear that the $(j,k)$-th entry of the matrix $D_i$ is equal to $C_{ik}^j$.
By the formula (\ref{riccAln}) we get

\begin{equation}\label{scalnilp}
\trace(\Ric^\mathfrak{n})=-\frac{1}{4}\sum_{i=1}^l\trace(D_i D_i')= -\frac{1}{4}\sum\limits_{i,j,k} (C_{ij}^k)^2.
\end{equation}

{\it Further in this section we consider the case $r = \trace(A^s A^s)>0$, i.~e. the operator $A$ is not skew-symmetric}.
We will prove the following

\begin{pred}\label{VRazAlg1}
Let $(\mathfrak{s},Q)$ be a solvable metric Lie algebra, $\mathfrak{n}=[\mathfrak{s},\mathfrak{s}]$, $\mathfrak{a}$ be
a $Q$-orthogonal complement to $\mathfrak{n}$ in $\mathfrak{s}$, $\dim(\mathfrak{a})=1$, $r = \trace(A^s A^s)>0$.
Then one of the following two exclusive assertions holds:

\begin{enumerate}
\item The ideal $\mathfrak{n}$ is commutative and for every nontrivial $X\in \mathfrak{a}$
the operator $\ad(X)|_{\mathfrak{n}}$ is trace-free, normal, but not skew-symmetric
with respect to $Q$ (in this case the Ricci operator of $(\mathfrak{s},Q)$
has only one negative eigenvalue, while all other eigenvalues are zero);
\item The Ricci operator of the metric Lie algebra $(\mathfrak{s},Q)$ has at least two negative eigenvalues.
\end{enumerate}
\end{pred}

First, we consider a matrix
\begin{equation}
L = \left( {{\begin{array}{*{20}c}
I  & \frac{1}{r} \cdot R_2 \\
0  & 1 \\
 \end{array} }} \right),
\end{equation}
where $I$ is the identity matrix, and the matrix

\begin{equation}
\overline{\Ric} =L \cdot \Ric \cdot L'=\left( {{\begin{array}{cc}
R_1 +\frac{1}{r} \cdot R_2 R_2' & 0 \\
0  & -r \\
 \end{array} }} \right).
\end{equation}

By the law of inertia for quadratic forms, the matrices $\overline{\Ric}$ and
${\Ric}$ have one and the same signature. But the matrix $\overline{\Ric}$
is block-diagonal (with a negative entry in the last block) and we immediately get

\begin{pred}\label{newpr1}
The operator ${\Ric}$ (for $r>0$) has at least two negative eigenvalues if and only if
the matrix
\begin{equation}\label{newmatr1}
R_1 +\frac{1}{r} \cdot R_2 R_2'=\Ric^\mathfrak{n} +
\frac{1}{2}[A,A'] - t A^s +\frac{1}{r} \cdot R_2 R_2'
\end{equation}
has at least one negative eigenvalue.
\end{pred}

Note that a symmetric matrix with negative trace has at least one negative eigenvalues.
Therefore, we get

\begin{pred}\label{newpr2}
If $r>0$ and
$$
\trace(\Ric^\mathfrak{n})-t^2+\frac{1}{r} \trace(R_2 R_2')<0,
$$
then the operator ${\Ric}$ has at least two negative eigenvalues.
\end{pred}

\begin{remark}\label{nonuniprop}
For non-unimodular Lie algebras, the inequality from the above proposition could be replaced by the inequality
\begin{equation}\label{ineqpr}
r\cdot \trace(\Ric^\mathfrak{n})+\trace(R_2 R_2')\leq 0,
\end{equation}
since $t>0$ for such algebras.
\end{remark}

\medskip

In the remainder of this section we {\it prove Proposition \ref{VRazAlg1}}.
We consider {\it two variants: {\bf $R_2=0$} and {\bf $R_2\neq 0$}.}

\begin{claim}\label{proofr20}
Proposition \ref{VRazAlg1} is valid for $R_2=0$.
\end{claim}

\begin{proof}
By Proposition \ref{newpr1} (and because of $R_2=0$), ${\Ric}$ has at least two negative eigenvalues if and only if
the matrix
\begin{equation}\label{vsvsvs}
\Ric^\mathfrak{n} + \frac{1}{2}[A,A'] - t A^s
\end{equation}
has at least one negative eigenvalue.
The trace of this matrix is (see (\ref{scalnilp}))
$$
-\frac{1}{4}\sum\limits_{i,j,k} (C_{ij}^k)^2-t^2
$$
(since $\trace([A,A'])=0$ and $\trace(A^s)=\trace(A)=t$).
If this trace is not zero, then ${\Ric}$ has at least two negative eigenvalues.

Now, suppose that this trace is zero. Then $t=0$ and $\mathfrak{n}$ is abelian
(this implies $\Ric^\mathfrak{n}=0$ in particular).
Further, if $A$ is not normal, then the operator (\ref{vsvsvs}), having the form
$\frac{1}{2}[A,A']$,
is trace-free and non-zero. Hence, it has at least one negative eigenvalue and
${\Ric}$ has at least two negative eigenvalues by Proposition \ref{newpr1}.
On the other hand, if $A$ is normal, then the operator
(\ref{vsvsvs}) is zero and ${\Ric}$ has only one negative eigenvalue, while all other its eigenvalues are zero.
This proves Proposition \ref{VRazAlg1} for $R_2=0$.
\end{proof}

Now, we consider a much more technically involved claim.

\begin{claim}\label{proofr2n0}
If $R_2\neq 0$ then the assertion (2) of Proposition \ref{VRazAlg1} is valid.
\end{claim}

\begin{proof}
Note that the value
$$ \trace(R_2 R_2')=\frac{1}{4}
\sum\limits_{i=1}^l (\trace(D_i \cdot A'))^2
$$
(see Remark \ref{eqeq}) does not depend on the choice of an orthonormal basis $\{e_i\}$,  $1 \leq
i \leq l$, in $\mathfrak{n}$. This assertion has been proved in \cite{NikCheb},
but we reproduce here a short argument for the convenience of the reader.
Consider another orthonormal basis
$\{\overline{e}_i\}$, $1 \leq i \leq l$, in $\mathfrak{n}$. Then $\overline{e}_i=\sum_j
q_{ji} e_j$ for all $i$, where $(q_{ji})$ is an orthogonal matrix. Therefore,
\begin{eqnarray*}
\overline{D}_i=\ad(\overline{e}_i)=\sum_j q_{ji} D_j,\\
\trace(\overline{D}_i A')=\sum_j q_{ji} \trace(D_j A'),\\
(\trace(\overline{D}_i A'))^2=\sum_{j,k} q_{ji} q_{ki} \trace(D_j
A')\trace(D_k A'),\\
\sum_i(\trace(\overline{D}_i A'))^2=\sum_{i,j,k} q_{ji} q_{ki}
\trace(D_j A')\trace(D_k A')= \\
\sum_{j,k}\left(\trace(D_jA')\trace(D_k A')\sum_{i} q_{ji} q_{ki}\right)=\\
\sum_{j,k}\ \delta_{jk} \trace(D_j A')\trace(D_k A')=
\sum_{j}(\trace(D_j A'))^2.
\end{eqnarray*}

Hence, we may use some special
basis $ \{e_i\}$ (in $\mathfrak{n}$) in order to get more suitable expressions for $R_2$ and $\trace(R_2 R_2')$.

In the linear space $\Der(\mathfrak{n}) \subset \End(\mathfrak{n})$
of derivations of $\mathfrak{n}$, we consider a
subspace $\InnDer(\mathfrak{n})$ of inner derivations.
We will use the projection
$P_{inner}: \Der(\mathfrak{n}) \rightarrow \InnDer(\mathfrak{n})$
(see (\ref{eqproj})).
Let us consider $\widetilde{A}=P_{inner}(A)$.

Let $\mathfrak{l}$ be a subspace of codimension $1$ in
$\mathfrak{n}$ such that for any $X\in \mathfrak{l}$
the inner derivation $\ad(X)$ lies in the orthogonal complement to
$\mathbb{R} \cdot \widetilde{A}$ with respect to inner product $\langle \cdot ,\cdot \rangle$.

Now we choose a
$Q$-orthonormal basis $\{e_i\}$ in $\mathfrak{n}$ such that
$e_i \in \mathfrak{l}$ for $i \geq 2$.
Hence, we get
$\langle D_i,A\rangle=\trace(D_i A')=0$ for $i\geq 2$.
Recall, that $i$-th entry of the column matrix $R_2$ is equal to
$-\frac{1}{2} \langle D_i, A \rangle$ (see Remark \ref{eqeq}). Since we suppose $R_2 \neq 0$,
then
$\langle D_1,A\rangle=\trace(D_1 A')\neq 0$.
Therefore,
\begin{equation}\label{bnonz1}
R'_2=\Bigl(-\frac{1}{2} \langle D_1, A \rangle,0,0,\dots,0\Bigr),
\end{equation}
\begin{equation}\label{bnonz2}
R_2R'_2=\diag\Bigl(\frac{1}{4} {\langle D_1, A \rangle}^2, 0,0,\dots,0 \Bigr),
\end{equation}
\begin{equation}\label{bnonz}
4 \trace(R_2 R_2')=(\trace(D_1 \cdot A'))^2={\langle D_1,A\rangle}^2.
\end{equation}
Now we will prove the inequality (\ref{ineqpr}) and study possibilities when it becomes an equality.

Recall that
$\trace(\Ric^\mathfrak{n})=-\frac{1}{4}\sum\limits_{i,j,k} (C_{ij}^k)^2$  by (\ref{scalnilp}),
On the other hand, by (\ref{bnonz}) and by the Cauchy–Bunyakovsky–Schwarz inequality we get
\begin{eqnarray*}
4\trace(R_2 R_2')=
{\langle D_1,A\rangle}^2 ={\langle D_1,\widetilde{A}\rangle}^2\leq {\langle D_1,D_1\rangle}
{\langle\widetilde{A},\widetilde{A}\rangle}=\\
{\langle\widetilde{A},\widetilde{A}\rangle}\trace(D_1 \cdot D_1')
={\langle\widetilde{A},\widetilde{A}\rangle}\sum_{j,k} (C^j_{1k})^2,
\end{eqnarray*}
where the inequality becomes an equality if and only if $\widetilde{A}=\lambda D_1$ for some $\lambda \in \mathbb{R}$
($D_1\neq 0$ by $R_2 \neq 0$, ${\langle D_1,\widetilde{A}\rangle}^2={\langle D_1,D_1\rangle}
{\langle\widetilde{A},\widetilde{A}\rangle}$ iff $\widetilde{A}$ is proportional to $D_1$).

Therefore,
\begin{eqnarray*}
r\cdot \trace(\Ric^\mathfrak{n})+\trace(R_2 R_2')=\\
-\frac{1}{4}\trace(A^s \cdot A^s)\sum\limits_{i,j,k} (C_{ij}^k)^2+
\frac{1}{4}{\langle D_1,\widetilde{A}\rangle}^2 = \\
-\frac{1}{4}\left({\langle A^s, A^s \rangle}\sum\limits_{i,j,k} (C_{ij}^k)^2-
{\langle D_1,\widetilde{A} \rangle}^2 \right)\leq \\
-\frac{1}{4}\left({\langle A^s, A^s \rangle}\sum\limits_{i,j,k} (C_{ij}^k)^2-
{\langle\widetilde{A},\widetilde{A}\rangle} {\langle D_1,D_1\rangle} \right)\leq \\
-\frac{1}{4}\left(\frac{1}{2}{\langle\widetilde{A},\widetilde{A}\rangle} \sum\limits_{i,j,k} (C_{ij}^k)^2-
{\langle\widetilde{A},\widetilde{A}\rangle} {\langle D_1,D_1\rangle} \right)= \\
-\frac{1}{8} {\langle\widetilde{A},\widetilde{A}\rangle} \left(\sum\limits_{i,j,k}
(C_{ij}^k)^2 - 2\sum_{j,k} (C^j_{1k})^2\right) \leq 0,
\end{eqnarray*}
since $C_{1j}^k=-C_{j1}^k$, $\sum\limits_{i,j,k}
(C_{ij}^k)^2 \geq 2 \sum\limits_{j,k} (C^k_{1j})^2$ and
$2{\langle A^s, A^s \rangle}\geq {\langle\widetilde{A},\widetilde{A}\rangle}$ by Proposition \ref{vspom1}.
If $\mathfrak{s}$ is non-unimodular (i.~e. $t=\trace(A)\neq 0$), then we
get that $\Ric$ has at least two negative eigenvalue
by Proposition \ref{newpr2} and Remark \ref{nonuniprop}.
\medskip

Now, we suppose that $\Ric$ {\it has at most one negative eigenvalue}.
This implies $t=\trace(A)=0$ and the equality $r\cdot \trace(\Ric^\mathfrak{n})+\trace(R_2 R_2')=0$.
From the above arguments we get that the latter equality holds if and only if
$2{\langle A^s, A^s \rangle}={\langle\widetilde{A},\widetilde{A}\rangle}$,
$\widetilde{A}=\lambda D_1$ for some $\lambda \in \mathbb{R}$ and $C_{ij}^k=0$ for $1 \not\in\{i,j\}$ simultaneously.
Then by Proposition \ref{vspom1} we see that $A-\widetilde{A}$ is a skew-symmetric derivation of $\mathfrak{n}$.
Since $A$ is not skew-symmetric, then $\widetilde{A}\neq 0$ and $\lambda \neq 0$.
If $C_{ij}^k=0$ for $1 \not\in\{i,j\}$, then $\Lin\{e_2,e_3,\dots,e_l\}$ is an abelian subalgebra of codimension $1$
in $\mathfrak{n}$. By Lemma \ref{vspom0} we get that $\Lin\{e_2,e_3,\dots,e_l\}$ is an ideal in
$\mathfrak{n}$. Therefore, $C_{ij}^1=0$ for all $i,j$.

Let us consider the matrix (\ref{newmatr1}) more closely, put
$$
\widetilde{R}:=\Ric^\mathfrak{n} +
\frac{1}{2}[A,A'] +\frac{1}{r} \cdot R_2 R_2'.
$$
By Proposition \ref{newpr1} it has no negative eigenvalue.
Since $\trace(\widetilde{R})=\trace(\Ric^\mathfrak{n})+1/r\cdot\trace(R_2 R_2')=0$, it means that
$\widetilde{R}$ is the zero matrix. Now we will prove, that the latter is {\it impossible} (under the above conditions).

{\it Let us suppose the contrary, i.~e. $\widetilde{R}=0$}.
At first, consider $\widetilde{R}_{11}$,
the $(1,1)$-th entry of $\widetilde{R}$.
By (\ref{riccAln}) we get
$$\Ric^{\mathfrak{n}} = -\frac{1}{2}
\sum\limits_{i=1}^{l} D_i^{\prime}D_i +\frac{1}{4} \sum\limits_{i=1}^{l}D_iD_i^{\prime}.
$$
Since $C_{ij}^1=0$ for all $i$ and $j$, the first column of the matrix $D_1$ and
the first row of every matrix $D_i$ ($1\leq i \leq s$) are zero. Therefore,
$\left(D_1^{\prime}D_1 \right)_{11}=\left(D_iD_i^{\prime} \right)_{11}=0$ for all $1\leq i \leq s$ and
$$
\left(D_i^{\prime}D_i \right)_{11}=\sum\limits_{j=1}^l (C_{i1}^j)^2=\sum\limits_{j=1}^l (C_{1i}^j)^2
$$
for $i \geq 2$. Therefore,
$\left( \Ric^{\mathfrak{n}}\right) _{11}=-\frac{1}{2}\sum\limits_{i,j=1}^l (C_{1i}^j)^2=
-\frac{1}{2}\langle D_1, D_1 \rangle$.

Since $\widetilde{A}=\lambda D_1$ and $A-\widetilde{A}$ is skew-symmetric, then
$A^s=\widetilde{A}^s=\lambda D_1^s$.
Since $D_1$ is a nilpotent operator, then by Lemma \ref{vnilp} we get
$$
r=\langle A^s, A^s \rangle = \lambda^2 \langle D_1^s,D_1^s \rangle =\frac{\lambda^2}{2} \langle D_1,D_1 \rangle.
$$
By (\ref{bnonz2})  we get
$(R_2R'_2)_{11}=\frac{1}{4}{\langle D_1, A \rangle}^2=\frac{1}{4}{\langle D_1, \widetilde{A} \rangle}^2=
\frac{1}{4}{\langle D_1, \lambda D_1 \rangle}^2=\frac{\lambda^2}{4} {\langle D_1, D_1 \rangle}^2$.
Hence,
$$
\left(\frac{1}{r}R_2R'_2\right)_{11}=\frac{1}{2}\langle D_1, D_1 \rangle
$$
and
\begin{equation}\label{aazero}
\widetilde{R}_{11}=-\frac{1}{2}\langle D_1, D_1 \rangle+\frac{1}{2}([A,A'])_{11}+\frac{1}{2}\langle D_1, D_1 \rangle=
\frac{1}{2}([A,A'])_{11}=0.
\end{equation}
Now we multiply the matrix equality
$$
\Ric^\mathfrak{n} +
\frac{1}{2}[A,A'] +\frac{1}{r} \cdot R_2 R_2'=0
$$
by the matrix $[A,A']$ from the right and calculate the traces of both sides:
$$
\trace(\Ric^\mathfrak{n}\cdot [A,A'])+\frac{1}{2}\trace([A,A']\cdot [A,A'])+
\frac{1}{r} \trace(R_2 R_2' \cdot [A,A'])=0.
$$
Recall that
$R_2R'_2=\diag\Bigl(\frac{1}{4} {\langle D_1, A \rangle}^2, 0,0,\dots,0 \Bigr)$ by (\ref{bnonz2}) and
$([A,A'])_{11}=0$ by (\ref{aazero}), hence, $\trace(R_2 R_2' \cdot [A,A'])=0$. By Proposition
\ref{riccider} we get the inequality \linebreak
$\trace(\Ric^\mathfrak{n}\cdot [A,A'])\geq 0$, that becomes an equality if and only if $A'\in \Der (\mathfrak{n})$.
Since $\trace([A,A']\cdot [A,A'])=\langle [A,A'],[A,A'] \rangle \geq 0$, then $[A,A']=0$,
$\trace(\Ric^\mathfrak{n}\cdot [A,A'])=0$,
and $A'\in \Der (\mathfrak{n})$.

By Lemma \ref{vspom1n} we get $A \in L_2$ (see the first section for the definitions of $L_i$).
Since $D_1 \in \InnDer(\mathfrak{n})$, then $D_1 \in L_3$ by Corollary
\ref{subder}. Finally, by Lemma \ref{vspom2} we get
$$
0=\langle A, D_1 \rangle = \langle \widetilde{A}, D_1 \rangle=\lambda \langle  D_1, D_1 \rangle.
$$
But this is impossible, since $\lambda \neq 0$ and $D_1 \neq 0$.
Therefore, $\widetilde{R}$ is not a zero matrix. This contradiction proves the claim.
\end{proof}

Therefore, we have proved
Proposition \ref{VRazAlg1}.

\section{Proof of the main results}

In this section we prove Theorem \ref{VRazAlg} in full generality.
Consider a subspace
$$
\widetilde{\mathfrak{a}}=\{X \in \mathfrak{a} \,|\, \ad (X)\,\, \mbox{is skew-symmetric in} \,\, (\mathfrak{s},Q)\}.
$$
There are three mutually exclusive cases:
$$
{\bf 1)}~\dim(\mathfrak{a})-\dim(\widetilde{\mathfrak{a}})\geq 2,\quad
{\bf 2)}~\dim(\mathfrak{a})-\dim(\widetilde{\mathfrak{a}})=1, \quad
{\bf 3)}~\dim(\mathfrak{a})=\dim(\widetilde{\mathfrak{a}}).
$$

{\bf Case 1)}. Choose a subspace $\mathfrak{b} \subset \mathfrak{a}$ such that
$\mathfrak{a}=\widetilde{\mathfrak{a}} \oplus \mathfrak{b}$. Then for every $X\in \mathfrak{b}$ the operator
$\ad(X)$ is not skew-symmetric and by Proposition \ref{kossim} we get $\Ric(X,X) <0$.
Since $\dim (\mathfrak{b}) \geq 2$, then the operator $\Ric$ has at least two  negative eigenvalues
(see Corollary \ref{dopmat}).
\smallskip

{\bf Case 2)}. We can choose a basis $\{f_1,f_2,\dots, f_m\}$ such that $f_i \in \widetilde{\mathfrak{a}}$ for $i \geq 2$.
Since the operator $\ad(f_i)$, $2 \leq i \leq m$, is skew-symmetric, then the matrix
$A_i$ are skew-symmetric and $B_i=0$ for $i\geq 2$ (see (\ref{adad})). Then $B_1=0$ also.
By (\ref{nnteor}) we get
$ \Ric = \left( {{\begin{array}{*{20}c}
 R_1 \hfill & R_2 \hfill\\
 R'_2 \hfill & R_3 \hfill\\
\end{array} }} \right),$
where
$$
R_1 = \Ric^\mathfrak{n} + \frac{1}{2}[A_1,A'_1] - t A_1^s,
$$
$$
R_2 = -\frac{1}{2} \left( \sum\limits_{i=1}^{l}D'_iC_i \right), \quad
R_3 = -L,
$$
where $\Ric^\mathfrak{n}$ is the matrix of the Ricci operator of the metric Lie algebra
$(\mathfrak{n}, Q|_{\mathfrak{n}})$
in the basis $\{e_1,....,e_l\}$,
$L$ is a $(m \times m)$-matrix with elements $l_{pq}= \trace(A_p^sA_q^s)$.
Since $A_i$ is skew-symmetric for $i \geq 2$, then
$L=\diag \bigl(\trace(A_1^s\cdot A_1^s),0,0,\dots,0\bigr)$.

Now, let us consider $\widehat{\mathfrak{s}}=\Lin(\mathfrak{n},f_1 )\subset \mathfrak{s}$.
It is clear that $\widehat{\mathfrak{s}}$ is closed under the Lie multiplication $[\cdot,\cdot]$, i.~e.
$\widehat{\mathfrak{s}}$  is a subalgebra of  the Lie algebra $\mathfrak{s}$.
Supply it with the inner product $Q|_{\widehat{\mathfrak{s}}}$. Let $\widehat{\Ric}$ be a matrix
of the Ricci operator of metric Lie algebra $(\widehat{\mathfrak{s}},Q|_{\widehat{\mathfrak{s}}})$ in the basis
$\{e_1,e_2, \dots, e_l,f_1\}$. Using the formula (\ref{nnteor}) one more time, we see that
$\widehat{\Ric}$ is submatrix of $\Ric$ corresponding to rows and columns with the numbers $1,2,\dots,l,l+1$.
By Proposition \ref{VRazAlg1}, we have two possibilities
for the metric Lie algebra  $(\widehat{\mathfrak{s}},Q|_{\widehat{\mathfrak{s}}})$:

{\bf 2a)} the ideal $\mathfrak{n}$ is commutative and
the operator $\ad(f_1)|_{\mathfrak{n}}$ is trace-free, normal, but not skew-symmetric
with respect to $Q|_{\mathfrak{n}}$;

{\bf 2b)} the Ricci operator $\widehat{\Ric}$
of the metric Lie algebra $(\widehat{\mathfrak{s}},Q|_{\widehat{\mathfrak{s}}})$ has at least two negative eigenvalues.

If 2a) holds, then for the metric Lie algebra $({\mathfrak{s}},Q)$
the possibility (2) of Theorem~\ref{VRazAlg} holds.

If 2b) holds, then the matrix $\Ric$ has at least two negative eigenvalues, since this property has its submatrix
$\widehat{\Ric}$ (see Corollary \ref{dopmat}). Hence, we prove Theorem \ref{VRazAlg} in case~2).

Before studying the case 3), we prove the following

\begin{lem}\label{nilp}
If $\mathfrak{s}$ is a non-abelian nilpotent Lie algebra, then the cases 2a) and 3) are impossible
for the metric Lie algebra
$(\mathfrak{s},Q)$ .
\end{lem}

\begin{proof} At first, prove that the case 2a) is impossible.
Suppose the contrary. Then (in the above notations)
the operator $\ad(f_1)|_{\mathfrak{n}}$ is trace-free and normal, but not skew-symmetric.
On the other hand, it is nilpotent, but
the only nilpotent normal operator is the zero operator.
We get a contradiction, since $\ad(f_1)|_{\mathfrak{n}}$ is not skew-symmetric.

Now, we prove that the case 3) is impossible.
Suppose the contrary. Every operator $\ad (f_i)$ are
both skew-symmetric (hence, normal) and nilpotent. Therefore, $\ad (f_i)$ is the zero operator for all $i=1,2,\dots,l$,
and $\mathfrak{a}$ lies in the center of the Lie algebra $\mathfrak{s}$. Hence,
$$
\mathfrak{n}=[\mathfrak{s},\mathfrak{s}]=[\mathfrak{n}\oplus \mathfrak{a},\mathfrak{n}\oplus \mathfrak{a}]=
[\mathfrak{n},\mathfrak{n}].
$$
But this is impossible, since $\mathfrak{n}$ is nilpotent and $\mathfrak{n} \neq 0$.
\end{proof}

\begin{remark}\label{last}
Lemma \ref{nilp} implies that for a non-abelian nilpotent metric Lie algebra $(\mathfrak{s},Q)$,
 only the cases 1) and 2b) are possible.
In both this cases {\it the Ricci operator of $(\mathfrak{s},Q)$
has at least two negative eigenvalue.  This proves Theorem \ref{NilpAlg}}.
\end{remark}
\smallskip

Finally, we consider the {\bf case 3)}.
Since all operators $\ad(f_i)$, $1 \leq i \leq m$, are skew-symmetric, then (for all $i$) the matrix
$A_i$ is skew-symmetric, $B_i=0$ and $A_i^s=0$ (see~(\ref{adad})).
By (\ref{nnteor}) we get
$ \Ric = \left( {{\begin{array}{*{20}c}
 R_1 \hfill & R_2 \hfill\\
 R'_2 \hfill & R_3 \hfill\\
\end{array} }} \right),$
where
$$
R_1 = \Ric^\mathfrak{n},
\quad
R_2 = -\frac{1}{2} \left( \sum\limits_{i=1}^{l}D'_iC_i \right), \quad
R_3 = 0,
$$
and $\Ric^\mathfrak{n}$ is the matrix of the Ricci operator of the metric Lie algebra
$(\mathfrak{n}, Q|_{\mathfrak{n}})$
in the basis $\{e_1,....,e_l\}$.

We have two possibilities: {\bf 3a)} $\mathfrak{n}$ is abelian; {\bf 3b)} $\mathfrak{n}$ is non-abelian.

If 3a) holds, then for the metric Lie algebra $({\mathfrak{s}},Q)$
the possibility (3) of Theorem~\ref{VRazAlg} holds.

If 3b) holds, then by Corollary \ref{dopmat} the matrix $\Ric$
has at least two negative eigenvalues, since this property has its submatrix
$\Ric^\mathfrak{n}$
(the latter is a statement of  Theorem~\ref{NilpAlg}, that we have proved in Remark \ref{last}).
Hence, we have proved Theorem \ref{VRazAlg} in full generality.

\vspace{10mm}

\bibliographystyle{amsunsrt}

\vspace{5mm}

\end{document}